\documentclass[reqno]{amsart}
\usepackage{amscd,amsfonts,amssymb}
\usepackage{graphicx}
\usepackage{color}

\numberwithin{equation}{section}

\newtheorem{Theorem}{Theorem}[section]
\newtheorem{Lemma}{Lemma}[section]

\theoremstyle{definition}

\theoremstyle{remark}
\newtheorem{Remark}{Remark}[section]

\author{ Pierre Degond}
\address{Pierre Degond, Universit\'e de Toulouse; UPS, INSA, UT1, UTM ;
Institut de Math\'ematiques de Toulouse ;
F-31062 Toulouse, France
\& CNRS; Institut de Math\'ematiques de Toulouse UMR 5219 ;
F-31062 Toulouse, France.}
 \curraddr{} \email{pierre.degond@math.univ-toulouse.fr}

\author{Hailiang Liu}
\address{Hailiang Liu, Department of Mathematics\\Iowa State University\\Ames, IA 50011.
} \curraddr{} \email{hliu@iastate.edu, Fax:(515)294-5454}

\title[]{Kinetic models for polymers with inertial effects}
\thanks{ }
\keywords{Polymers, kinetic description, Brownian forces,  Rod like
models, Dumbbell model}
\date{January 10, 2009}
\begin{document}
\begin{abstract}
Novel kinetic models for both Dumbbell-like and rigid-rod like
polymers are derived, based on the probability distribution function
$f(t, x, n, \dot n)$ for a polymer molecule positioned at $x$ to be
oriented along direction $n$ while embedded in a $\dot n$
environment created by inertial effects. It is shown that the
probability distribution function of the extended model, when
converging, will lead to well accepted kinetic models when inertial
effects are ignored such as the Doi models for rod like polymers,
and the Finitely Extensible Non-linear Elastic (FENE) models for
Dumbbell like polymers.
\end{abstract}
\maketitle
\tableofcontents

\section{Introduction}

In this paper we derive novel kinetic models for both Dumbbell like
and rigid-rod like polymers in the presence of inertial forces. The
model is to describe dynamics of the probability distribution
function embedded in high dimensional configuration space due to
inertial effects. We then prove that the limit equation of the new
model when inertial force vanishes leads to current models with no
inertial effects such as the FENE model and the Doi model,
respectively. This illustrates consistency of our kinetic models
with existing models.

The wide range of applications of polymer materials has attracted
new areas of academic and industrial research. The synthesis of
different type of polymers has enlarged the range of applications of
polymer materials to areas where mechanical properties are
important. Materials made up of macro-molecules such as polymers
display properties that completely differ from those made from
small molecules. The description of polymer dynamics is often based
on large assemblies of molecules, the characteristics could be modeled
in terms of their statistical properties.

Most polymers are long chains or branches of repeated chemical
units. The full description of each atom in the polymer by molecular
dynamics is not feasible for the huge computational effort. Coarse
grained models are often expected with macroscopic space and time
properties of complex fluids. Typical models such as bead-spring
chain for flexible polymers and the rigid rod model for liquid
crystalline polymers have been established by the pioneers in
polymer science.

In general, the flow modeling of polymers has to take into account
the internal structure, characterized by both positional and
orientational order of phases. Such incorporation is often done by
adding new balance equations to those that govern structure-less
Newtonian fluids. These new balances must be evaluated from the
behavior of polymers. According to the relative size of the bending
persistence size and the length of the polymer, two canonical types
of polymers are widely studied: the Dumbbell model and the rigid-rod
like model. In modeling motion of polymers, it is essential to
explore an accurate method of solutions of the Langevin equation for
particles undergoing Brownian movement (rotational or translational)
under the influence of external fields. A large number of reviews,
text books and monographs on the theory, applications and rheology
of polymeric materials have appeared in the literature, see e.g.
\cite{Fl:1953, Fl:1969, He:1976, Fe:1980, Doi:1986, DE:1986,
BCAH:1987, GP:1993, Oe:1996, La:1999, Rd:2002}.

There are three main levels of description of polymeric fluids:
atomistic modeling, kinetic modeling \cite{DE:1986, He:1976}, and
the macroscopic approach of continuum mechanics \cite{BCAH:1987}. We
shall exploit the kinetic approach. Models of kinetic theory provide
a coarse-grained description of molecular configurations wherein
atomistic processes are ignored altogether (Doi and Edwards
\cite{DE:1986}, Bird et al \cite{BCAH:1987}, and
Ottinger\cite{Oe:1996}). Kinetic theory models for polymer solutions
are most naturally exploited numerically by means of stochastic
simulation or Brownian dynamics methods \cite{Oe:1996}.  A kinetic
theory model when equipped with an expression relating stress to
molecular configurations plays an important role in developing
micro-macro methods of computational rheology \cite{OP:2002,
Ke:2004}. In current kinetic theory models for polymers, the inertia
of molecules is often neglected. However, neglect of inertia in some
cases leads to incorrect predictions of the behavior of polymers.
The forgoing considerations indicate that the inertial effects are
of importance in practical applications, e.g., for short time
characteristics of materials based on the relevant underlying
phenomena.

It is thus the goal of this paper to model dynamics of the density
distribution of polymers when the inertial force is no longer
ignorable. More precisely we shall be particularly interested in
modeling two canonical types of polymers: Dumbbell-like and rod-like
polymers, which when inertial forces are not considered have been
well understood. We first derive kinetic models including inertial
effects from particle dynamics (continuum limit in the Brownian
motion), we then show the limit of the augmented models when
inertial forces vanish leads to the inertia-free model.

We now summarize our main results for two types of polymers.
\subsection{Dumbbell-like polymers} A mcromolecule is idealized as an
`elastic dumbbell' consisting of two `beads' joined by a spring
which can be modeled by an end-to-end vector $n$. Here $n$ is in a
bounded ball $B(0, n_0)$, which means that the extensibility of the
polymers is finite. Let $f(t, x, n, p, q)$ denote the distribution
function of Dumbbell-like polymers on the space variables $x\in
\mathbb{R}^d$, $d=2, 3$, the translational velocity $p\in
\mathbb{R}^d$, the end-to-end vector $n$ as well as the
orientational velocity $q \in \mathbb{R}^d$. And $t$ is the time.
The novel kinetic model to be derived is
\begin{align}\label{fp2-}
\partial_t f +\nabla_x\cdot (pf) &+\nabla_n\cdot(qf)
+\nabla_p\cdot(-\frac{\zeta}{m}(p-u(x))f)\\ \notag
&+\nabla_q\cdot\left(\left(-\frac{\zeta}{m}(q-n\cdot \nabla_x u(x))
-\frac{2F}{m}\right)f\right)=\frac{2k_BT\zeta}{m^2}[\Delta_p f
+\Delta_q f],
\end{align}
where $\zeta$ is the frictional coefficient for the beads with mass
$m$, $k_B$ is the usual Boltzmann constant,  $T$ is the absolute
temperature, and $F$ is the spring force between beads. The force
usually derives from a potential, and has different forms for
different models. For the well known  FENE potential
$$
F=\frac{Hn}{1-n^2/n^2_0},
$$
where $H$ is a spring constant \cite{BCAH:1987}.

The above model under the scaling
$$ m=\epsilon^2, \quad \sqrt{m}p
\to p, \quad \sqrt{m}q \to q,
$$
leads to
\begin{align} \label{fe-}
\epsilon^2 \partial_t f & +\epsilon \nabla_x\cdot(pf)  + \epsilon
\nabla_n\cdot(qf) \notag \\
\quad & +\epsilon \nabla_p \cdot (\zeta u(x)f)+\epsilon
\nabla_q\cdot[(\zeta n\cdot \nabla_x u -2F)f]=Q(f),
\end{align}
where
$$
Q(f):=\zeta \nabla_p \cdot(pf +2k_BT \nabla_p f)+\zeta \nabla_q
\cdot(qf +2k_BT \nabla_q f).
$$

Our result for Dumbbell like polymers could thus read as follows:
\begin{Theorem}The limit $\epsilon\to 0$ of $f$ is given by $f^0=\rho M$
where $\rho=\rho(t, x, n) \geq 0$  and $M$ are given by
$$
\rho(t, x, n)=\int f^0(t, x, n, p, q)dpdq, \quad
M=\exp\left(-\frac{p^2+q^2}{4k_BT}\right).
$$
Furthermore, $\rho(t, x, n)$ satisfies the following kinetic
equation
\begin{equation}\label{fp1-}
    \partial_t \rho + \nabla_x\cdot(u(x)\rho) +\nabla_n \cdot ((n\cdot \nabla_x u-\frac{2F}{\zeta})\rho)=
\frac{2k_BT}{\zeta}\Delta_x \rho +\frac{2k_BT}{\zeta}\Delta_n \rho.
\end{equation}
\end{Theorem}

\subsection{Rod-like polymers}
Polymers are idealized as rods of fixed length. The orientation
space is $n \in \mathbb{S}^{d-1}$. Let $f(t, x, n, p, \omega)$
denote the distribution function of rod-like polymers on the space
variables $x\in \mathbb{R}^d$, the translational velocity $p\in
\mathbb{R}^d$, the orientational vector $n$ as well as the angular
velocity $\omega$. Here $\omega$ is on the tangent bundle
$T_n\mathbb{S}^{d-1}$. And $t$ is the time.  Our rescaled kinetic
model can be formulated as
\begin{align} \label{ske-}
\epsilon \partial_t f &+\nabla_x \cdot(pf)+ \mathcal{R}\cdot(\omega
f)+ \zeta_t
\nabla_p \cdot(u(x) f)\notag \\
\quad & + \nabla_\omega \cdot((\zeta_r n\times \nabla_xu \cdot
n-\mathcal{R}\cdot U) f)=\frac{1}{\epsilon}Q(f),
\end{align}
where
$$
Q(f)=\zeta_t \nabla_p\cdot(pf+k_BT \nabla_p f)+\zeta_r \nabla_\omega
\cdot(f\omega +k_BT \nabla_\omega f).
$$
Here $\epsilon$ denotes the inertial parameter similar to
(\ref{fe-}), $u$ is the fluid velocity, and $U$ is certain
interaction potential of rods. $\mathcal{R }=n\times \nabla_n$ is
the rotational gradient operator, and $k_B, T$ denote the Boltzmann
constant and absolute temperature, respectively. $\zeta_t, \zeta_r$
are frictional coefficients in $x$ and $n$ directions.

Our result for rod-like polymers then reads as follows:
\begin{Theorem}The formal limit $\epsilon \to 0$ of $f$ is given by $f^0=\rho M$
where $\rho=\rho(t, x, n) \geq 0$  and $M$ are given by
$$
\rho(t, x, n)=\int f^0(t, x, n, p, q)dpd_n\omega, \quad
M=\exp\left(-\frac{p^2+\omega^2}{2k_BT}\right).
$$
Furthermore, $\rho(t, x, n)$ satisfies the kinetic equation
\begin{equation}\label{fp4-}
  \partial_t \rho +\nabla_x\cdot(u(x)\rho) +\mathcal{R}\cdot(n\times \nabla_xu \cdot \rho)=D_t\Delta_x \rho
  +D_r \mathcal{R}\cdot \left[\mathcal{R}\rho +\frac{\rho}{k_B T}\mathcal{R}U \right]
\end{equation}
where
$$
D_t=\frac{k_BT}{\zeta_t}, \quad D_r=\frac{k_BT}{\zeta_r}.
$$
\end{Theorem}

Our derivation of these kinetic models is based on establishing
motion laws of polymer molecules, followed by a conversion into the
kinetic description. The formal limit when inertia vanishes is
justified by taking the classical approach for hydrodynamic limits.
To this end a rescalling is adopted, so that the collision operator
is set on fast scale, and the dissipation of the collision operator
drives the states to the unique equilibrium states $M$.  Derivation
of models for Dumbbell-like polymers and the formal limit
justification are given in \S 3, and those for rod-like polymers are
given in \S 4. Some concluding remarks are presented in \S 5.

Finally, we wish to close this section by pointing to a vast body of
recent work on mathematical treatment of kinetic theory models for
polymers,  their constitutive models as well as their coupling with
fluid models (so called micro-macro models), see e.g.
\cite{Rm91,DLP02,ELZ02,CKT04,CKT04+,ELZ04,JLL04,Cp05,CTV05,
CV05,FS05,DLY05,LZZ05,EZ06,JLO06,ZZ06,CFTZ07,Li07,LLZ07,LM07,CM08,LL08,LL08+,LZ08,M08,OT08,
ZWFW05,ZZZ08} and references therein. These should  provide guidance
in establishing the various levels of mathematical theory of kinetic
models developed in this work.


\section{Kinetic models for Dumbbell polymers }
\subsection{Equations of motion with non-trivial inertia}

We consider a polymer consisting of two-beads connected by one-spring.
Each bead as a coarse grained particle represents several chemical units and experiences four
kinds of forces in the dilute case where there is no interaction for inter--and intra-dumbbells.

Assume the two beads are positioned at $x_1$ and $x_2$, the Langevin equation of the beads are just balance of different forces expressed as
\begin{align*}
m_1\ddot{x_1} &=-\zeta_1(\dot x_1 -u(x_1))+F_1 +\sigma_1 \dot W_1,\\
m_2\ddot{x_2} &=-\zeta_2(\dot x_2 -u(x_2))+F_2 +\sigma_2 \dot W_2.
\end{align*}
Here $W_i$ are independent standard Brownian motions, $\zeta_i$ are frictional coefficients for the $i$-th bead with mass $m_i$, $\sigma_i=\sqrt{2k_BT\zeta_i}$ by fluctuation-dissipation theorem, where $k_B$ is the Boltzmann
constant, and $T$ is the temperature. The spring forces $F_1=-F_2=F$ by Newton's Third Law
depends only on the end-to-end vector $n$.

A natural configuration for this underlying polymer includes both
position of the polymer with these two beads connected to one spring
$$
x=\frac{x_1+x_2}{2},
$$
and the end-to-end vector denoted by $n=x_2-x_1$. We assume that $m_1=m_2=m$ and $\zeta_1=\zeta_2=\zeta$. Then the difference  and average of the above two equations gives a coupled system
\begin{align}\label{2.1}
m\ddot {x} & =-\zeta(\dot x-u(x))+\sqrt{4k_BT\zeta}\dot W_1,\\ \label{2.2}
 m\ddot {n} &=-\zeta(\dot n-n \cdot \nabla_x u)-2F+\sqrt{4k_BT\zeta}\dot W_2.
\end{align}
The approximation utilizes the equation $u(x)\sim (u(x_1)+u(x_2))/2$, $u(x_2)-u(x_1) \sim \nabla_x u \cdot n$, and the two new Brownian motions
$
(W_2 \pm W_1)/\sqrt{2}
$
(still denoted by $W_i$).\footnote{Here we use the fact that for any $\alpha \in[0, 1]$, $\alpha W_1(t)\pm
\sqrt{1-\alpha^2}W_2(t)$ remains a Brownian motion.}

If the inertia were ignored for this two beads spring system, we
would have the following stochastic equations (SDE),
\begin{eqnarray}\label{xn1}
               \dot x & =&u(x) +\sqrt{\frac{4k_BT}{\zeta}}\dot
               W_1,\\ \label{xn2}
               \dot n & =&n\cdot \nabla_x u -\frac{2F}{\zeta}+\sqrt{\frac{4k_BT}{\zeta}}\dot W_2.
             \end{eqnarray}
The differential operator or the infinitesimal generator corresponding to the process $(x(t), n(t))$
is given by
\begin{equation}\label{L1}
L\equiv \langle u(x), \nabla_x\rangle + \langle n\cdot \nabla_x u
-\frac{2F}{\zeta}, \nabla_n\rangle +\frac{2k_BT}{\zeta}(\Delta_x
+\Delta_n),
\end{equation}
where
$$
\nabla_x=\left(\frac{\partial}{\partial_{x_1}}, \cdots, \frac{\partial}{\partial_{x_d}}\right), \quad
\nabla_n=\left(\frac{\partial}{\partial_{n_1}}, \cdots, \frac{\partial}{\partial_{n_d}}\right),
$$
with $\langle, \rangle$ denoting the Euclidean inner product in
$\mathbb{R}^d$. The configuration of Dumbbell like polymers can be
illustrated in Figure \ref{fig1}:
\begin{figure}[pth]
\begin{center}

\includegraphics[scale=0.4]{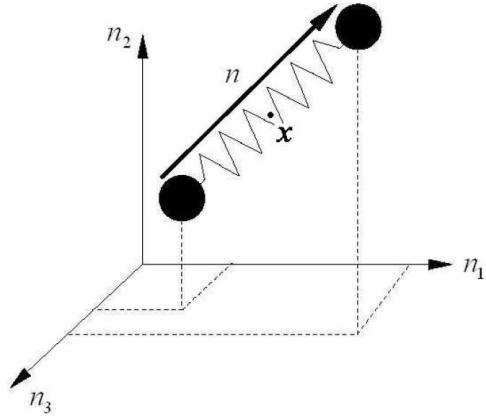}
\caption{A Dumbbell like polymer}\label{fig1}
\end{center}
\end{figure}

When inertial force is significant, one has to take the inertial effects into consideration.
We now rewrite the system (\ref{2.1})-(\ref{2.2}) into a first-order system
\begin{align*}
\dot x & =p,\\
\dot n &=q,\\
\dot p &=-\frac{\zeta}{m}(p-u(x))+\frac{\sqrt{4k_BT\zeta}}{m}\dot W_1,\\
\dot q &= -\frac{\zeta}{m}(q-n\cdot \nabla_x u(x)) -\frac{2F}{m}+
\frac{\sqrt{4k_BT\zeta}}{m}\dot W_2.
\end{align*}
This system can be considered as a degenerate stochastic differential equation with a singular diffusion matrix of $4d \times 4d$. The differential operator corresponding to the process $(x(t), n(t), p(t), q(t))$
is also uniquely defined. In fact for the above dissipative stochastic differential equation with well-defined initial data, a unique solution $(x, n, p, q)(t)$ up to an explosion time is ensured if  $u(x)$  and $F$  are smooth functions of the configuration variables.

\subsection{Kinetic description}
We now formulate the kinetic description of the above motion laws.
For a given instant $t$, the random variables $X=(x, n, p, q)$ 
can be characterized by a probability density function (PDF) $f(t, x, n; p, q)$ defined by
$$
f(t, x, n;p, q,)dX=P[\{X\leq \tilde X\leq X+dX\}].
$$
In words, the right-hand side is the probability that the random variable 
$\tilde X$ falls between the sample space values 
$X$ and $X+dX$ for different realizations of the polymer motion. Thus the above statistical ODE system can be converted into a PDE of the form
\begin{align}\label{fp2}
\partial_t f +\nabla_x\cdot (pf) &+\nabla_n\cdot(qf)
+\nabla_p\cdot\left(-\frac{\zeta}{m}(p-u(x))f \right)\\ \notag
&+\nabla_q\cdot\left(\left(-\frac{\zeta}{m}(q-n\cdot \nabla_x u(x))
-\frac{2F}{m}\right)f\right)=\frac{2k_BT\zeta}{m^2}[\Delta_p f
+\Delta_q f].
\end{align}
It is known that in the case with no inertial forces the kinetic
equation is also of the Fokker-Planck type. The model follows from
the motion law (\ref{xn1})-(\ref{xn2}) as follows:
\begin{equation}\label{fp1}
    \partial_t \rho + \nabla_x\cdot(u(x)\rho) +\nabla_n \cdot \left(\left(n\cdot \nabla_x u-\frac{2F}{\zeta}\right)\rho\right)=
\frac{2k_BT}{\zeta}\Delta_x \rho +\frac{2k_BT}{\zeta}\Delta_n \rho,
\end{equation}
where $\rho(t, x, n)$ denotes the corresponding probability density
function.

\begin{Remark}
The polymer scale is much smaller than the fluid scale, the
diffusion in space is often ignorable. Here we make no such
distinction.
\end{Remark}

It is natural to ask whether the effective dynamics in (\ref{fp2}) is dominated by (\ref{fp1}) when inertia vanishes. As a first step we shall identify an effective equation satisfied
by the limit of $f$ when $m \downarrow 0$, if the limit exists. In the remainder, we will always suppose that $f$ is as smooth as needed.
\subsection{Scaling}
We make the following scaling
$$
m=\epsilon^2, \quad \sqrt{m}p \to p, \quad \sqrt{m}q \to q,
$$
to obtain
\begin{align} \label{fe}
\epsilon^2 \partial_t f +\epsilon \nabla_x\cdot(pf)+ \epsilon \nabla_n\cdot(qf)+\epsilon
\nabla_p \cdot (\zeta u(x)f)+\epsilon \nabla_q\cdot[(\zeta n\cdot
\nabla_x u -2F)f]=Q(f),
\end{align}
where
$$
Q(f):=\zeta \nabla_p \cdot(pf +2k_BT \nabla_p f)+\zeta \nabla_q
\cdot(qf +2k_BT \nabla_q f).
$$
The remaining task in this section is to justify the following
\begin{Theorem}
The limit $\epsilon \downarrow 0$ of the $f^\epsilon$ is given by $f^0=\rho(t, x, n)M(p, q)$, where
\begin{equation}\label{m}
M=\exp\left(-\frac{p^2+q^2}{4k_BT}\right)
\end{equation}
satisfying $Q(M)=0$ and $\rho(t, x, n)$ solves the kinetic equation (\ref{fp1}).
\end{Theorem}
As usual the limit equation that is associated with (\ref{fe}) does
not depend on details of the operator $Q$. Rather, it depends on $Q$
possessing certain properties related to conservation, dissipation,
and equilibria that are stated below.  Note the operator $Q$ acts on
the augmented variables $(p, q)$ only and leaves other variables
$(t, x, n)$ as parameters. Thus we list some properties of $Q$ as an
operator acting on functions of $(p, q)$ only, which are essential
in the derivation of the limit equation.
\subsection{Properties of $Q$}
Note that the operator $Q(f)$ has the following properties
\begin{itemize}
\item The operator can be written as a conservative form
$$
Q(f)=2 \zeta k_BT \nabla_{p, q}\cdot \left(M \nabla_{p, q}
\left(\frac{f}{M}\right)\right).
$$

\item The conservation form leads to
 $$\int Q(f)dpdq=0$$
 for every $f\in L^1(dpdq)$.  This relation expresses the physical laws of mass conservation
 during the action of $Q$.

Moreover, this is the only such conservation law. This means that
$$
\int Q(f)\psi dpdq=0, \quad \forall f \Leftrightarrow \psi \in {\rm
span} \{1\}.
$$
\item Dissipation. There is a nonnegative function $\eta(f)$ that is
an entropy for the operator $Q$. This means that for $\eta(f) \in
L^1(dpdq)$ we have
$$
\int \eta'(f) Q(f) \leq 0,
$$
whose vanishing characterizes the local equilibrium of $Q$. For
flexible polymers, $\eta'(f)=f/M$ since
$$
D^Q(f)=\int Q(f)\frac{f}{M}dpdq=-\int M\left( \nabla_{p, q}
\left(\frac{f}{M}\right)\right)^2dpdq \leq 0.
$$
\item Existence of equilibria.  From the above dissipation property we see that
the function $f(p, q)$ such that $Q(f)=0$ forms a space identified through
$$
 \frac{f}{M}\in {\rm span}\{1\}.
$$
This equilibrium is unique up-to a constant multiplier.  We
normalize the density function then
$$
f_{eq}=\left[\int Mdpdq \right]^{-1}M.
$$


\end{itemize}
\subsection{Asymptotic limit as $\epsilon \downarrow 0$}
Equipped with these properties we now investigate the asymptotic limit.
Define
\begin{align*}
\rho^\epsilon(t, x)& =\int f dpdq,\\
J_1^\epsilon(t, x)& =\frac{1}{\epsilon}\int p f dpdq,\\
J_2^\epsilon(t, x)& =\frac{1}{\epsilon}\int q f dpdq.
\end{align*}
Then integration of (\ref{fe}) together with conservation property
of $Q$ yields
\begin{equation}\label{con}
    \partial_t \rho^\epsilon +\nabla_x\cdot J_1+\nabla_n \cdot J_2=0,
    \quad \forall \epsilon >0.
\end{equation}
We suppose that as $\epsilon \downarrow 0$ the limit of $f$ is identified  by
$f^0$. It follows from the equation (\ref{fe}) that
$$
f^0\in Q^{-1}(0).
$$
This implies that $f^0=M(p, q)\rho^0(t, x, n)$,  for the operator
$Q$ does not act on $(t, x, n)$.  Hence the probability density
function $f$ takes the following form
$$
f^\epsilon(t, x, n; p, q)=\rho^\epsilon(t, x, n)M(p, q)+\epsilon
g^\epsilon,
$$
where
$$
\lim_{\epsilon \downarrow 0} \rho^\epsilon(t, x, n)=\rho^0(t, x, n).
$$
Note that $\int pMdpdq=\int qMdpdq=0$, we thus have
\begin{align}
J_1^\epsilon &=\frac{1}{\epsilon}\int pf^\epsilon dpdq=\int
pg^\epsilon dpdq,\\
J_2^\epsilon &=\frac{1}{\epsilon}\int q f^\epsilon dpdq=\int q
g^\epsilon dpdq.
\end{align}
The limit of these flux functions depends only on the limit of $g^\epsilon$.

Linearity of the operator $Q$ and $Q(M)=0$ leads to
$$
Q(f^\epsilon)=Q(\rho^\epsilon(t, x)M(p, q)+\epsilon
g^\epsilon)=\epsilon Q(g^\epsilon).
$$
Using the equation (\ref{fe}) we obtain
\begin{align*}
Q(g^\epsilon) &=\frac{1}{\epsilon} Q(f^\epsilon)\\
& =\epsilon \partial_t f^\epsilon +
\nabla_x\cdot(pf^\epsilon)+\nabla_n\cdot(qf^\epsilon)+
\nabla_p\cdot(\zeta u(x)f^\epsilon)+ \nabla_q\cdot((\zeta n\cdot
\nabla_x u-2F)f^\epsilon)\\
& =\nabla_x\cdot(pM\rho^\epsilon )+\nabla_n\cdot(qM\rho^\epsilon)+
\nabla_p\cdot(\zeta u(x) \rho^\epsilon M)\\
& \qquad + \nabla_q\cdot((\zeta n\cdot \nabla_x u-2F)M\rho^\epsilon)
+O(\epsilon).
\end{align*}
Let $g^\epsilon \to g^0$ as $\epsilon \to 0$, then it is clear that
$$
Q(g^0)=p\cdot \nabla_x \rho^0 M+ q\cdot \nabla_n \rho^0 M +\zeta
\rho^0 u(x) \cdot \nabla_p M +\rho^0 (\zeta n\cdot \nabla_xu -2F)\cdot
\nabla_q M.
$$
Again note that  $Q$ is a linear operator, we may assume the
ansatz for $g^0$ as follows.
$$
g^0=a\cdot \nabla_x \rho^0 +b\cdot \nabla_n \rho^0+\zeta \rho^0
c\cdot u(x) +\rho^0 (\zeta n\cdot \nabla_xu -2F)\cdot d.
$$
In comparison with the expression of $Q(g^0)$ above, we see that it
is sufficient to find $a, b, c, d$ such as
\begin{align*}
Q(a)& =pM, \\
Q(b)&=qM,\\
Q(c)&=\nabla_p M,\\
Q(d) &=\nabla_q M.
\end{align*}
These relations are understood in the sense that the operator acts
on each component of the underlying vector. In order to identity
$(a, b, c, d)$ we state the following lemma:
\begin{Lemma} Let $Q$ be the collision operator and $g \in L^2(R^p\times R^q)$ such that
$\int g dpdq=0$. The problem

\begin{equation}\label{fg}
    Q(\psi)=g,
\end{equation}
has a unique weak solution in $\Lambda=\{\psi, \quad \int \psi
dpdq=0, \quad \psi \in  H^1 \}. $
\end{Lemma}
This lemma,  following from applying the Lax-Milgram theorem to the
corresponding variational formulation, ensures that there exists
unique $a, b, c, d$ in the space $\{\phi|\int \phi dpdq=0\}$ since
$$
\int \{pM, qM, \nabla_p M, \nabla_q M\}dpdq=0.
$$
We also note that from
$$
\nabla_p M= -\frac{1}{2k_BT}pM, \quad \nabla_q M= -\frac{1}{2k_BT}qM,
$$
it follows
$$
c=-\frac{a}{2k_BT}, \quad d=-\frac{b}{2k_BT}.
$$
It remains to determine only $a$ and $b$. Let $a=\alpha Mp$ we have
$$
Q(a)=-\alpha \zeta pM=pM,
$$
which gives $\alpha =-\frac{1}{\zeta}$, leading to  $a=-\frac{1}{\zeta}pM$. Similarly we obtain $b=-\frac{1}{\zeta} qM$.

Clearly we see that
$$
\int \{b, d\}\cdot pdpdq=0, \quad \int \{a, c\}\cdot qdpdq=0.
$$
These enable us to determine the asymptotic limits of the fluxes.
\begin{align*}
J_1^0 &= \int pg^0 dpdq \\
& = \int (a\cdot\nabla_x \rho^0) p  +\int (b\cdot \nabla_n \rho^0 )p +
\zeta  \rho^0 \int (c \cdot u(x)) p
+\rho^0 \int d\cdot (\zeta n \cdot  \nabla_x u-2F)p\\
& =\int a\cdot \left(\nabla_x \rho^0- \frac{\zeta u(x)\rho^0}{2k_B T} \right) p dpdq \\
& =-\left(\nabla_x \rho^0- \frac{\zeta u(x)\rho^0}{2k_B T} \right)\frac{2k_B T}{\zeta}\int Mdpdq.
\end{align*}
\begin{align*}
J_2^0 &= \int qg^0 dpdq \\
& = \int (a\cdot \nabla_x \rho^0)q  +\int b\cdot \nabla_n \rho^0  q + \zeta \rho^0 \int (c \cdot u(x)) q
+\rho^0 \int d\cdot (\zeta n\cdot \nabla_x u-2F) q\\
& =\int b \cdot\left(\nabla_n \rho^0- \frac{\rho^0}{k_B T}(\zeta n\cdot \nabla_x u-2F) \right) q dpdq \\
& = -\left(\nabla_n \rho^0- \frac{\rho^0}{2k_B T}(\zeta n\cdot \nabla_x u-2F) \right) \frac{2k_B T }{\zeta}\int Mdpdq.
\end{align*}
Now the equation (\ref{con}) divided  by $\int Mdpdq$  asymptotically converges to
$$
\partial_t \rho^0 + \nabla_x\cdot(u(x)\rho^0) +\nabla_n \cdot ((n\cdot \nabla_x u-\frac{2F}{\zeta})\rho^0)
= \frac{2k_BT}{\zeta}\Delta_x \rho^0 +\frac{2k_BT}{\zeta}\Delta_n \rho^0.
$$
This is exactly the equation (\ref{fp1}) derived from ignoring inertial forces.

\begin{Remark} With these kinetic equations for Dumbell like polymers, it is
natural to understand how these polymers contribute to the
macroscopic flow governed by a coupled Navier-Stokes system
$$
\partial_t u +(u\cdot \nabla_x u)u+\nabla_x p= \frac{\gamma}{Re}\Delta_x u
+\frac{1-\gamma}{ReDe}\nabla_x \cdot \tau, \quad \nabla_x \cdot u=0,
$$
with the stress determined by
$$
\tau=\int (n\otimes F)fdndpdq.
$$
Here $De$ is called Deborah number, which is the most important
parameter in non-Newtonian fluids. $Re$ and $\gamma$ are the
Reynolds number and viscosity ratio, respectively.  The tensor force
follows the case when inertial force is ignored,  the derivation of
the tensor force with inclusion of inertial effects is beyond the
scope of this work.
\end{Remark}
\begin{Remark}
The spring force is determined by a potential function $U$ through $F=-\nabla_n U$. Different potential leads to different models. Two choices are commonly used: \\

\begin{tabular}{|l|c|c|}
  \hline
   & Force  & Potential \\
  Hookean spring & $Hn$& $\frac{1}{2}Hn^2$ \\
  FENE spring & $\frac{Hn}{1-(n/n_0)^2}$ & $-\frac{1}{2}Hn_0^2ln\left(1-(n/n_0)^2 \right)$ \\
  \hline
\end{tabular}\\

where $n_0$ is the maximum extension of the beads connector.
\end{Remark}

\section{Rigid rod-like polymers}
Though many polymers are flexible, there is still a large class of
polymers which are not flexible and assume a rod-like structure.
Rod-like polymers have some peculiar properties and have attracted
a great deal of attention.

We consider rod-like molecules in concentrated regime. Rod-like polymers can have only two kinds of motion, i.e., translation and rotation. The translational Brownian motion is the random motion of the position vector $x$ of the center of mass, and the rotational Brownian motion is the random motion of the unit vector $n$ ($|n|=1$) which is parallel to the polymer.  We shall build a kinetic model for the probability distribution of orientational motions of rod in every point of phase space $(x, p)$. This serves as a microscopic equation, which is expected to be coupled with the macroscopic equation (the Navier-Stokes equation) for the fluid velocity.

For the convenience of calculations in what follows, we introduce a
local coordinate on the sphere $\mathbb{S}^2$ as $q=(\theta, \phi)$,
and set $n=(sin \theta cos \phi, sin \theta sin \phi, cos \theta)$
with $\theta \in [0, \pi)$ and $\phi\in [0, 2\pi)$. The unit vector
$e_\theta=(cos \theta cos \phi, cos \theta sin \phi, -sin \theta)$
and $e_\phi=(-sin \phi, cos \phi, 0)$. We note that $\frac{\partial
n}{\partial \theta}=e_\theta$, $\frac{\partial n}{\partial \phi}=sin
\theta e_\phi$. Thus, any tangent vector on the sphere is written as
$$
\dot n =\dot \theta e_\theta +\dot \phi sin \theta e_\phi,
$$
which gives
$$
\omega=n\times \dot n= \dot \theta e_\phi - sin \theta \dot \phi
e_\theta.
$$
 The gradient is
$$
\nabla_n=e_\theta \partial_\theta +\frac{e_\phi}{sin
\theta}\partial_\phi.
$$
The rotational gradient is
\begin{equation}\label{go}
\mathcal{R}= n\times \nabla_n=\left(-\frac{cos \phi cos \theta}{sin
\theta} \partial_\phi -sin \phi \partial_\theta, -\frac{cos \theta
sin \phi}{sin \theta} \partial_\phi +cos\phi \partial_\theta,
\partial_\phi \right)^T.
\end{equation}
Divergence of a vector $A= A_\theta e_\theta +  A_\phi  e_\phi$ is
$$
\nabla_n\cdot A=\frac{1}{sin \theta} \partial_\theta(sin \theta
A_\theta) +\frac{1}{sin \theta}\partial_\theta A_\phi.
$$
The gradient in terms of $\omega$ is defined only for a given $n$,
and will be understood from now on as
$$
 n\cdot \nabla_\omega f=0, \quad \forall f.
$$

To be more specific, we regard the identical liquid crystal
molecules as inflexible rods of a thickness $b$ which is much
smaller than their length $L$, as illustrated in Figure \ref{fig2}:
\begin{figure}[pth]
\begin{center}
\includegraphics[scale=0.4]{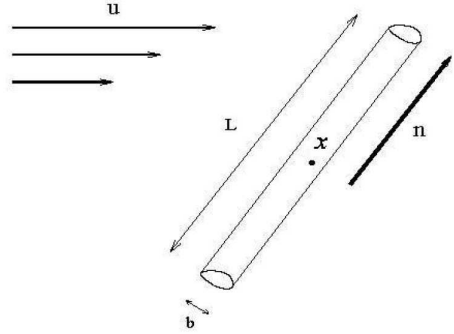}
\caption{A rod like polymer}\label{fig2}
\end{center}
\end{figure}

\subsection{Translational Brownian motion}
Let $U$ be an interaction potential, the force thus  induced is $-\nabla_xU$. The motion law is
$$
m\ddot{x}=-\zeta_t(\dot x -u(x))-\nabla_x U,
$$
where $\zeta_t$ is the friction coefficient and $u(x)$ is the fluid
velocity field. If $f$ is the probability distribution in $x$, the
Brownian force is expressed as $F_B=-\nabla_x (k_BT ln f)$, where
$-k_BTln f$ is the chemical potential. For nontrivial inertial
force, the distribution needs to be accounted in an extended
environment with inclusion of $p=\dot x$. The corresponding
translational Brownian force thus reduces to
\begin{equation}\label{cp}
F_B=\nabla_p W, \quad W=-\frac{\zeta_t k_B T}{m}ln f .
\end{equation}
This can be justified by a similar derivation based on Brownian
motions as that in Section 2. The scaled coefficient reflects
balances between the friction force and the inertial force. Putting
together we have the following translational motion law
\begin{align}
\dot x &=p, \\
\dot p &= -\frac{\zeta_t}{m}(p-u(x)) -\frac{1}{m}\nabla_x U-\frac{\zeta_t }{m^2} \nabla_p (k_B Tln f).
\end{align}
In what follows we shall conveniently use the chemical potential to describe the Brownian force.

\subsection{Rotational Brownian motion}
We consider a rod rotating with angular velocity $\omega$, then the rotational motion can be described as
$$
[J]\dot \omega =T,
$$
where $J$ is the moment of inertia and $T$ is the total torque.

1. Rotational frictional force.

Consider a rod of length $L$  placed in a viscous fluid with fluid
velocity $u(x)$, where $x$ is the center of mass of the rod.

The rod is parameterized by $s$, ranging from
$-L/2$ to $L/2$, then the position vector of the $s$-point on this rod is written as
$$
x(s)=x+ ns, \quad x(0)=x, \quad -L/2\leq s\leq L/2.
$$
Let $v(s)$ and $F(s)$ be the velocity of this point, and forces acting on it. The velocity
$v(s)$ is expressed by the angular velocity $\omega$
$$
 v(s)=\omega \times (x(s)-x)=\omega \times (ns).
$$
The frictional force at $x(s)$ is
$$
F(s)=-\xi(s)(v(s)-u(x(s))),
$$
where $\xi(s)$ is the frictional coefficient, being symmetric $\xi(s)=\xi(-s)$.
Note that
$$
u(x(s))\simeq u(x)+\nabla_x u(x(s)-x(0))=u(x)+s \nabla_x u\cdot n.
$$
The frictional force thus reduces to
$$
F(s)=-s\xi(s)(\omega \times n -\nabla_x u \cdot n)+\xi(s)u(x).
$$
Thus the total torque induced by the frictional force acting on the rod  is
\begin{align*}
T& =\frac{1}{L}\int ^{L/2}_{-L/2}(x(s)-x)\times F(s)ds\\
 &=- \xi_2 n \times (\omega \times n -\nabla_x u \cdot
 n)+ \xi_1 n \times u(x),
\end{align*}
where
$$
\xi_k=\frac{1}{L}\int_{-L/2}^{L/2}s^k\xi(s) ds, \quad k=1, 2.
$$
Symmetry of $\xi(s)=\xi(-s)$ leads to $\xi_1=0$. Then using $n\cdot
\omega=0$ we have
$$
T=-\zeta_r\left[(Id-n\otimes n) \omega -n \times (\nabla_x u \cdot
n)\right]=-\zeta_r\left[\omega -n \times (\nabla_x u \cdot
n)\right],
$$
where $\zeta_r=\xi_2$ denotes the rotational friction coefficient.

2. Thermodynamic potential force. \\

From $[J]\dot \omega=T$ it follows
$$
\frac{d}{dt}E = T\cdot \omega, \quad E=\frac{1}{2} \omega^\top  [J]\omega.
$$
Let the potential be denoted by $U$. Then $d(E+U)=0$ gives
$dE=-dU=-\nabla_n U\cdot dn.$

Note that from $|n|=1$ we have
$$
\omega =n \times \dot n,
$$
which leads to
$$
\omega dt=n\times dn.
$$
We thus have
$$
T\cdot(n\times dn)=T\cdot \omega dt=dE=-\nabla_n U \cdot dn.
$$
Therefore $T\times n=-\nabla_n U$, i.e.,
\begin{equation}\label{ru}
    T=-\mathcal{R}U,
\end{equation}
where $ \mathcal{R}$ is the rotational gradient operator given by
(\ref{go}).

3. Rotational Brownian force. \\
The following property will be used on our derivation of forces
below.
\begin{Lemma} \label{lem3.1} For a fixed vector $a$, let $x=a\times y$. Then for
any function smooth $g$,
$$
\nabla_y g=-a\times \nabla_x g
$$
and
$$
a\cdot \nabla_y g=0, \quad \forall g.
$$
\end{Lemma}
\begin{proof}Let $\epsilon_{ijk}$  as the usual permutation symbol,
then $x_i=\sum_{j,k}\epsilon_{ijk} a_jy_k$.  We thus have
\begin{align*}
(\nabla_y g)_m &=\sum_i \partial_{x_i}g \frac{\partial x_i
}{\partial y_m}=\sum_i \partial_{x_i}g
\sum_{j,k}\epsilon_{ijk}a_j\delta_{mk}\\
&=\sum_{i,j}\epsilon_{mij}\partial_{x_i}g
a_j=-\sum_{i,j}\epsilon_{mji}a_j \partial_{x_i}g.
\end{align*}
This gives the desired relation. We now show the second claim.
$$
\sum_{m}a_m\partial_{y_m}g= \sum_{mij}\epsilon_{mij}\partial_{x_i}g
a_m a_j=-\sum_{mij}\epsilon_{mij}\partial_{x_i}g a_j a_m,
$$
which ensures that $a\cdot \partial_y f=0$ holds for any smooth
function $g$

\end{proof}

Consider a rod again with center of mass at $x$, and momentum $p=0$
at rest, the point vector of rod is $x(s)=x+sn$ with $s$ ranging in
$[-L/2, L/2]$. Correspondingly
$$
p(s)= s \dot n=s \omega \times n.
$$
From this we obtain that
\begin{equation}\label{os}
\omega=\frac{1}{s}n \times p(s).
\end{equation}
Also rod symmetry implies that
\begin{equation}\label{ps}
p(-s)=-p(s).
\end{equation}
For any fixed $s$ the associated Brownian force is calculated by
\begin{equation}\label{fb}
F_B(s) = \nabla_{p(s)}W, \quad   W=-\frac{\zeta_r k_B T}{m}ln f .
\end{equation}
Note that for each fixed $s$ and vector $n$, the mapping from $p(s)$
to $\omega$ is well-defined. The result in Lemma \ref{lem3.1} gives
\begin{align*}
F_B(s) = \nabla_{p(s)}W =  -\frac{1}{s} n\times \nabla_\omega W.
\end{align*}
Thus the corresponding Brownian torque is determined by
$$
T_B=\frac{1}{L}\int^{L/2}_{-L/2}s n \times F_B(s) ds,
$$
which when combined with the above calculations leads to
\begin{align*}
T_B&=\frac{1}{L}\int_{-L/2}^{L/2} s n \times \left(
-\frac{1}{s}n\times \nabla_\omega W \right)ds \\
& =-n\times (n\times \nabla_\omega W)\\
&=(Id-n\otimes n)\nabla_\omega W\\
&=-\frac{\zeta_r k_B T}{m}(Id-n\otimes n)\nabla_\omega (ln f)\\
&=-\frac{\zeta_r k_B T}{m}\nabla_\omega (ln f).
\end{align*}
Note that the moment of inertia for the rod is
$$
[J]=j(Id-n\otimes n), \quad j=\frac{mL^2}{12}.
$$
Using $\dot \omega \cdot n=0$, we thus have
$$
[J]\cdot \dot \omega=j \dot \omega.
$$
This together with all forces involved leads to the following motion
law
\begin{align}
\dot n &= \omega \times n, \\
\frac{mL^2}{12}\dot \omega &=-\zeta_r (\omega -n\times (\nabla_x
u)\cdot n)  -\frac{\zeta_r k_B T}{m}\nabla_\omega (ln f)
-\mathcal{R} U,
\end{align}
where $U$ is the interaction potential.
\begin{Remark} The potential can be defined as a generalized Onsager's
potential:
$$
U=\nu^2 k_B TbL^2 \int_{R_x}\int_{R_p}\int_{|n'|=1}\int_{n\cdot
\omega=0}B(x,p;x',p')|n \times n'|f(t, x',p', n',
\omega)dn'dx'dp'd_n\omega.
$$
Here $R_x$ and $R_p$ is the configuration space for variables $x$
and $p$. $B$ is some localized kernel.
\end{Remark}

\subsection{Kinetic equations}
We shall derive a kinetic equation in the phase space $(x, p, n,
\omega)$ with $(x, p)\in \mathbb{R}^3 \times \mathbb{R}^3$ and $(n,
\omega)\in  T\mathbb{S}^2$, where
\begin{equation}\label{mm}
    T\mathbb{S}^2:=\{(n, \omega)| \quad n\in \mathbb{S}^2, \quad \omega \in
    T_n\mathbb{S}^2\},
\end{equation}
which is usually called the tangent bundle of the manifold
$\mathbb{S}^2$.  We start from the continuity equation of a formal
form
\begin{equation}\label{fp3}
\partial_t f+\nabla_x \cdot (\dot x f) +\nabla_p \cdot (\dot p f)+\nabla_{n} \cdot (\dot
n f)+\nabla_{\dot n} \cdot (\ddot n f)=0.
\end{equation}
This equation can be simplified when restricted on the tangent
bundle: $(n, \omega)\in T\mathbb{S}^2 .$

First we state the following
\begin{Lemma}\label{lem3.2} Let $(n, \omega)$ be any element of the tangent bundle $T\mathbb{S}^2$, mapped to another element
$(m, \dot n)\in T\mathbb{S}^2$ such that
$$
m=n, \quad \dot n= \omega \times n,
$$
then for any smooth function $f$  we have
\begin{equation}\label{re}
\nabla_n f= \nabla_m f -\omega \times \nabla_{\dot n} f, \quad
\nabla_\omega f= n\times \nabla_{\dot n} f.
\end{equation}

\end{Lemma}
\begin{proof}Let $\epsilon_{ijk}$ be the usual permutation symbol,
then $\dot n_l= \sum_{jk}\epsilon_{ljk} \omega_jn_k$. Further we
have
\begin{align*}
\partial_{n_i}f &= \partial_{m_i}f +\sum_l \partial_{\dot n_l}f
\cdot \sum_{j, k}\epsilon_{ljk}\omega_j \delta_{ik}\\
&=\partial_{ m_i}f +\sum_{l,j}\epsilon_{ilj} \partial_{\dot n_l}f
\omega_j \\
&=\partial_{ m_i}f - \sum_{l, j}\epsilon_{ijl} \omega_j
\partial_{\dot n_l}f,
\end{align*}
where $\delta_{ij}$ is the Kronecker delta symbol. This gives the
first relation in (\ref{re}). The second relation follows from
\begin{align*}
\partial_{\omega_i}f &=\sum_l \partial_{\dot n_l}f
\cdot \sum_{j, k}\epsilon_{ljk}\delta_{ij} n_{k}\\
&=\sum_{l, k}\epsilon_{lik} \partial_{\dot n_l}f n_k.
\end{align*}
\end{proof}
Using similar arguments we have the following:
\begin{Lemma}\label{lem3.3} Let $g$ be any smooth vector function, then
\begin{equation}\label{re2}
    (n \times \nabla_n )\cdot g=-\nabla_n \cdot(n\times g), \quad
     (n \times \nabla_\omega )\cdot g=-\nabla_\omega \cdot(n\times
     g).
\end{equation}
\end{Lemma}
Note from $\omega=n\times \dot n$ we have $\dot \omega=n\times \ddot
n$ leading to $\dot \omega \cdot n=0$. Also
$$
\dot n=\omega \times n, \quad \ddot n=\dot \omega \times n
-|\omega|^2 n.
$$
Equipped with above lemmas and relations, we are able to reduce the
equation (\ref{fp3}). We only check the last two terms on the
left-hand of (\ref{fp3}). First,
$$
\nabla_n\cdot(\dot n f)=\nabla_n\cdot(\omega \times n f)=\mathcal{R}\cdot(\omega
f),
$$
where $\mathcal{R}:=n\times \nabla_n$ is the rotational gradient
operator. For any smooth function $g$, Lemma 3.2 and Lemma 3.3 imply
that
$$
\nabla_{\dot n}g=-(n\times \nabla_\omega) g.
$$
This enables us to simplify the last term in equation (\ref{fp3})
\begin{align*}
\nabla_{\dot n} \cdot (\ddot n f)& =-(n\times \nabla_\omega)\cdot ((\dot \omega
\times n -|\omega|^2 n)f)\\
&=\nabla_\omega \cdot (n\times (\dot \omega \times n
-|\omega|^2n)f)\\
&=\nabla_\omega\cdot((Id-n\otimes n)\dot \omega f)\\
&=\nabla_\omega\cdot(\dot \omega f).
\end{align*}
Therefore the effective kinetic equation becomes
\begin{equation}\label{ke}
\partial_t f+\nabla_x \cdot (pf) +\nabla_p \cdot (\dot p f)+\mathcal{R} \cdot (\omega f)
+\nabla_{\omega} \cdot (\dot \omega f)=0,
\end{equation}
where
\begin{align*}
\dot p & =-\frac{\zeta_t}{m}(p-u(x))-\frac{1}{m}\nabla_x
U-\frac{\zeta_t}{m^2}\nabla_p (k_BT ln f),\\
 \dot \omega &= -\frac{\zeta_r}{m} (\omega -n\times \nabla_x u(x) \cdot
n) -\frac{1}{m} \mathcal{R }U -\frac{\zeta_r}{m^2}\nabla_\omega(k_BT
ln f).
\end{align*}
Note here the coefficient $L^2/12$ has been absorbed in both the
$\zeta_r$ and the potential $U$, which is independent of both  $p$
and $\omega$.

\subsection{Scaling} We are interested in the solution behavior
when inertia vanishes.

We now make the following scaling
$$
m=\epsilon^2, \quad \epsilon p \to p, \quad \epsilon \omega \to
\omega,
$$
under this scaling the system (\ref{ke}) written into the new
variables becomes
\begin{align} \label{ske}
\epsilon \partial_t f &+\nabla_x \cdot(pf)+ \mathcal{R}\cdot(\omega f)+ \zeta_t
\nabla_p \cdot(u(x) f)\notag \\
\quad & + \nabla_\omega \cdot((\zeta_r n\times \nabla_xu \cdot
n-\mathcal{R}\cdot U) f)=\frac{1}{\epsilon}Q(f),
\end{align}
where
$$
Q(f)=\zeta_t \nabla_p\cdot(pf+k_BT \nabla_p f)+\zeta_r \nabla_\omega
\cdot(f\omega +k_BT \nabla_\omega f).
$$
Our next task is to investigate the formal limit $\epsilon \to 0$ of this problem.
\begin{Theorem}
The limit $\epsilon \downarrow 0$ of the $f^\epsilon$ is given by $f^0=\rho(t, x, n)M(p,\omega)$, where
\begin{equation}\label{m_rod}
M=\exp\left(-\frac{p^2+\omega^2}{2k_BT}\right)
\end{equation}
satisfying $Q(M)=0$ and $\rho(t, x, n)$ solves the following kinetic equation
\begin{equation}\label{fp4}
  \partial_t \rho +\nabla_x\cdot(u(x)\rho) +\mathcal{R}\cdot(n\times \nabla_xu \cdot \rho)=D_t\Delta_x \rho +D_r \mathcal{R}\cdot \left[\mathcal{R}\rho +\frac{\rho}{k_B T}\mathcal{R}U \right],
\end{equation}
where
$$
D_t=\frac{k_BT}{\zeta_t}, \quad D_r=\frac{k_BT}{\zeta_r}.
$$
\end{Theorem}
\begin{Remark}
(i)
When the coefficient $D_t=0$, the equation (\ref{fp4}) is the Doi's
kinetic model for rod-like polymers with inertial forces ignored.
This equation is also called the Smouluchowski equation in literature. \\
(ii) If the translation diffusion has different strengths and $U$ also depends on $x$, one needs to change
$D_t\Delta_x \rho$ to
$$
\nabla_x\cdot\left\{[D_{\shortparallel}n\otimes n+D_{\perp}(Id-n\otimes n)]
[\nabla_x \rho +\frac{\rho}{k_BT}\nabla_x U]
\right\}.
$$
The above justification procedure remains valid.
\end{Remark}
\subsection{Properties of $Q$}
The operator $Q$ here acts on the augmented variables $(p, \omega)$
only and leaves other variables $(t, x, n)$ as parameters. Thus we
list some properties of $Q$ as an operator acting on functions of
$(p, \omega)$ only, which are essential in the derivation of the
limit equation.

Set
$$
M=\exp\left(-\frac{p^2 +\omega^2}{2k_BT} \right)
$$
with $\omega$ satisfying $\omega \cdot n=0$ for any fixed $n \in \mathbb{S}^2$.

This form enables us to conclude the following
\begin{Lemma}\label{qq}
The operator $Q$ has the following properties:\\
\begin{itemize}
\item [(i)]The operator $Q$ can be written as
\begin{equation}\label{qf}
Q(f):=\zeta_t k_BT\nabla_p\cdot \left(M\nabla_p\left( \frac{f}{M}
\right)\right)+\zeta_rk_BT\nabla_\omega\cdot
\left(M\nabla_\omega\left( \frac{f}{M}\right)\right).
\end{equation}
\item[(ii)] Conservation. For every $f\in L^1(dpd_n\omega)$,
 $$\int Q(f)dp d_n \omega=0.$$
 Moreover, this is the only such conservation laws. In other words,
for any $f\in L^1(dpd_n\omega)$,
$$ \int Q(f)\psi dpd_n\omega=0, \quad \forall f \Leftrightarrow
\psi \in {\rm span} \{1\}.
$$
\item[(iii)] Dissipation. There is a nonnegative function $\eta(f)$
that is an entropy for the operator $Q$. This means that for
$\eta(f) \in L^1(dpd_n\omega)$ we have
$$
D^Q(f)=\int Q(f) \eta'(f) dpd_n\omega \leq 0.
$$
In the case of rod-like polymers,  $\eta'(f)=f/M$ and the
dissipation production is
$$
-k_BT \left[ \zeta_t \int M\left(\nabla_p\left(\frac{f}{M}\right)
\right)^2dp d_n\omega +\zeta_r \int M
\nabla_\omega\left(\frac{f}{M}\right)^2 dpd_n\omega \right]\leq 0.
$$

\item[(iv)] Existence of equilibria.  From the above dissipation
property we see that the function $f(p, \omega)$ such that $Q(f)=0$
if and only if
$$\frac{f}{M}\in  {\rm span}\{1\}.$$

\end{itemize}
\end{Lemma}
\noindent {\bf Proof of Lemma \ref{qq}}. (i) The conservation
follows from a direct calculation. (ii) The conservation form in (i)
leads to the null integration. Since the $ \int Q(f)\psi
dpd_n\omega=0$ holds true for any $f$, we let $f=\psi M$ to have
$$
\int Q(\psi M)\psi dpd_n\omega=0.
$$
The left-hand side of this relation is non-positive as
$$
0=-k_BT \left[ \zeta_t \int M\left(\nabla_p \psi \right)^2dp
d_n\omega +\zeta_r \int M (\nabla_\omega\psi)^2dpd_n \omega
\right]\leq 0.
$$
Thus $\psi$ must satisfy both $\nabla_p \psi=0$ and $\nabla_\omega
\psi =0$. Then one has $n \cdot \nabla_{\omega} \psi=0$. These
ensure that $\psi$ must be independent of $(p, \omega)$.

(iii) Dissipation property comes from integration by parts on the
tangent plane.

(iv) From the dissipation inequality we see that if $Q(f)=0$, then
$D^Q(f)=0$. The non-negativity of $D^Q$ yields
$$
\frac{f}{M} \in {\rm span}\{1\}.
$$

\subsection{Asymptotic limit as $\epsilon \downarrow 0$ }
We now investigate the formal limit of the problem, assuming all
involved functions are smooth and convergence holds true as needed.

We suppose that $f \to f^0$ as $\epsilon \to 0$. Then, from the scaled kinetic equation (\ref{ske}), $Q[f]=O(\epsilon)$ and we deduce that $Q(f^0)=0$.

By property (ii) we have $f^0=\rho M$, with $\rho=\rho(t, x, n) \geq
0$ and $n\in \mathbb{S}^2$. Note that $Q$ acts only on $(p,
\omega)$, $\rho$ are functions of $(t, x, n)$.   To find this
dependence, we use the generalized collisional invariants. We
integrate the equation with respect to $(p, \omega)$ to find the
continuity equation
$$
\partial_t \rho^\epsilon +\nabla_x \cdot J_1^\epsilon +\mathcal{R} \cdot J_2^\epsilon=0, \quad \forall \epsilon>0,
$$
where the density and fluxes are defined by
\begin{align*}
\rho^\epsilon(t, x, n)& =\int f dpd_n\omega,\\
J_1^\epsilon(t, x, n)& =\frac{1}{\epsilon}\int p f dpd_n\omega,\\
J_2^\epsilon(t, x, n)& =\frac{1}{\epsilon}\int \omega f dpd_n\omega.
\end{align*}
Here the integration over $\omega$ is interpreted as for any fixed
$n$ with $(n, \omega)\in T\mathbb{S}^2$. In the limit $\epsilon \to
0$, $\rho^\epsilon \to C \rho^0$ and $J_i^\epsilon \to C J_i^0$
($i=1,2$) with $C=\int Mdpd_n\omega$, and we obtain
\begin{equation}\label{limit}
  \partial_t\rho^0 +\nabla_x \cdot J_1^0 +\mathcal{R} \cdot J_2^0=0.
\end{equation}
Thus we may assume that
$$
f^\epsilon=\rho^\epsilon(t, x, n)M(p, \omega) +\epsilon g^\epsilon(t, x, p, n, \omega),
$$
which leads to
$$
Q(f^\epsilon)=\epsilon Q(g^\epsilon).
$$
To determine the limiting flux we need to explore the limit of $g^\epsilon$. Using the kinetic equation (\ref{ske}) we have
\begin{align*}
Q(g^\epsilon) &=\frac{1}{\epsilon} Q(f^\epsilon)\\
& =\epsilon \partial_t f^\epsilon +
\nabla_x\cdot(pf^\epsilon)+\mathcal{R} \cdot(\omega f^\epsilon)+
\nabla_p\cdot(\zeta_t u(x)f^\epsilon)+ \nabla_\omega \cdot((\zeta_r
n\cdot \nabla_x
u-\mathcal{R }U)f^\epsilon)\\
& =\nabla_x\cdot(pM\rho^\epsilon )+ \mathcal{R }\cdot(\omega
M\rho^\epsilon)+
\nabla_p\cdot(\zeta_t u(x) \rho^\epsilon M)\\
& \qquad + \nabla_\omega \cdot((\zeta_r n\cdot \nabla_x u-\mathcal{R} U)M\rho^\epsilon)
+O(\epsilon).
\end{align*}
Let $g^\epsilon \to g^0$ as $\epsilon \to 0$, then it is clear that
$$
Q(g^0)=p\cdot \nabla_x \rho^0 M+ \omega \cdot \mathcal{R} \rho^0 M +\zeta_t
\rho^0 u(x) \cdot \nabla_p M + \rho^0 (\zeta_r n \times (\nabla_xu \cdot n) -\mathcal{R}U)\cdot
\nabla_\omega M.
$$
Again note that  $Q(f)$ is a linear operator, we may assume the
ansatz for $g^0$ as follows.
$$
g^0=a\cdot \nabla_x \rho^0 +b\cdot \mathcal{R }\rho^0+\zeta_t \rho^0
c\cdot u(x) +\rho^0 (\zeta_r  n\times \nabla_x u \cdot n -\mathcal{R}U)\cdot d.
$$
In comparison with the expression of $Q(g^0)$ above, we see that it
is sufficient to find $a, b, c, d$ such as
\begin{align*}
Q(a)& =pM, \\
Q(b)&=\omega M,\\
Q(c)&=\nabla_p M,\\
Q(d) &=\nabla_\omega M.
\end{align*}
In order to uniquely determine $\{a, b, c, d\}$  we state the
following
\begin{Lemma}\label{lem3.5} Let $g\in L^2(dpd_n\omega)$ such that $\int g
dpd_n\omega=0$. The problem
\begin{equation}\label{lax}
    Q(\psi)=g,
\end{equation}
has a unique weak solution in the space $H^1(dpd_n\omega)$.
\end{Lemma}
\begin{proof}For each fixed $n\in \mathbb{S}^2$,
we apply the Lax-Milgram theorem to the following variational
formulation of (\ref{lax}):
\begin{equation}\label{var}
k_BT\int \left[\xi_t M\nabla_p \psi \cdot \nabla_p \phi +\xi_r M
\nabla_\omega\psi \cdot  \nabla_\omega \phi \right]dpd_n\omega=\int
g\phi dpd_n\omega
\end{equation}
for all $\phi\in H^1(dpd_n\omega)$. The function $M$ is bounded from
above and below on $\mathbb{R}_p^d \times T_n(\mathbb{S}^2)$, so the
bilinear form at the left-hand side is continuous and coercive on
$H^1(dpd_n\omega)$. The right-hand side is a continuous linear form
on $H^1(dpd_n\omega)$ due to the zero average of $g$ over
$\mathbb{R}_p^d \times T_n(\mathbb{S}^2)$.
\end{proof}

Lemma \ref{lem3.5}  ensures that there exists unique $a, b, c, d$ in
the space $\{\phi|\int \phi dpd_n\omega =0\}$ since
$$
\int \{pM, \omega M, \nabla_p M, \nabla_\omega M\}dpd_n\omega=0.
$$
Note also that from
$$
\nabla_p M= -\frac{1}{k_BT}pM, \quad \nabla_\omega M= -\frac{1}{k_BT}\omega M,
$$
it follows
\begin{equation}\label{cd}
    c=-\frac{a}{k_BT}, \quad d=-\frac{b}{k_BT}.
\end{equation}
It remains to determine only $a$ and $b$. Let $a=\alpha Mp$ we have
$$
Q(a)=-\alpha \zeta_t pM=pM,
$$
which gives $\alpha =-\frac{1}{\zeta_t}$, leading to
\begin{equation}\label{a}
a=-\frac{1}{\zeta_t}pM.
\end{equation}
Similarly let $b=\beta M\omega$, then
\begin{align*}
Q(b) &= \zeta_rk_BT\nabla_\omega \cdot (\beta M)\\
&=\beta \zeta_rk_BT \nabla_\omega M \\
&=-\beta \zeta_r  M \omega,
\end{align*}
which together with $Q(b)=M\omega$ gives $\beta=-1/\zeta_r$, thus
\begin{equation}\label{b}
b=-\frac{1}{\zeta_r} \omega M.
\end{equation}
It is straightforward to verify that
$$
\int \{b, d\}\cdot pdpd_n\omega =0, \quad \int \{a, c\}\cdot \omega dp d_n\omega=0.
$$
These enable us to determine the asymptotic limits of the fluxes.
\begin{align*}
J_1^0 &= \int pg^0 dpd_n\omega \\
& = \int (a\cdot\nabla_x \rho^0) p  +\int (b\cdot R \rho^0 )p +
\zeta_t  \rho^0 \int (c \cdot u(x)) p
+\rho^0 \int d\cdot (\zeta_r n\cdot  \nabla_x u-\mathcal{R}U )p \\
& =\int a\cdot \left(\nabla_x \rho^0- \frac{\zeta_t u(x)\rho^0}{k_B T} \right) p dpd_n\omega,
\end{align*}
\begin{align*}
J_2^0 &= \int \omega g^0 dpd_n\omega \\
& = \int (a\cdot \nabla_x \rho^0)\omega  +\int b\cdot R \rho^0  \omega
+ \zeta_t \rho^0 \int (c \cdot u(x)) \omega
+\rho^0 \int d\cdot (\zeta_r n\cdot \nabla_x u-\mathcal{R}U) \omega\\
& =\int b \cdot\left(\mathcal{R} \rho^0- \frac{\rho^0}{k_B T}(\zeta_r n\cdot \nabla_x u-\mathcal{R}U) \right) \omega dpd_n\omega.
\end{align*}
In order to further simplify the above fluxes, we state the
following lemma:
\begin{Lemma}\label{jj}
Consider the space $p\in \mathbb{R}^3$ and $(n, \omega)\in
T\mathbb{S}^2$.
For any vector $A$ we have
\begin{align}
\int M(p\cdot A)pdpd_n\omega &= k_BT A\int M dp d_n\omega,\\
\int M(\omega \cdot A)\omega dp d_n\omega &= k_BT (Id-n\otimes n)A\int Mdp d_n\omega.
\end{align}
\end{Lemma}
Apply this lemma to the above expressions to obtain
\begin{align}
J_1^0&= [\rho^0 u(x)-\frac{k_BT}{\zeta_t}\nabla_x \rho^0]\int Mdp d_n\omega,\\
J_2^0 &=(Id-n\otimes n)\left[-\frac{k_BT}{\zeta_r} \mathcal{R}\rho^0 +\frac{1}{\zeta_r}  \left( \
\zeta_r n\times \nabla_x u\cdot n-\mathcal{R}U \right)\right]\int Mdp d_n\omega.
\end{align}
A simple calculation shows for arbitrary $A$ that
$$
(Id-n\otimes n)(n\times A)=n\times A.
$$
Therefore
$$
J_2^0=\left[-\frac{k_BT}{\zeta_r} \mathcal{R}\rho^0 + \left( \
 n\times \nabla_x u\cdot n-\frac{1}{\zeta_r}  \mathcal{R}U \right)\right]\int Mdp d_n\omega.
$$
Now the limiting equation (\ref{limit}) divided  by $C=\int Mdpd_n\omega$  reduces to
$$
\partial_t \rho^0 + \nabla_x\cdot(u(x)\rho^0) +\mathcal{R}\cdot ((n\times  \nabla_x u \cdot n-\frac{\mathcal{R}U}{\zeta_r})\rho^0)=
\frac{k_BT}{\zeta_t}\Delta_x \rho^0 +\frac{k_BT}{\zeta_r}\mathcal{R}\cdot \mathcal{R} \rho^0.
$$
Regrouping with $\rho^0$ replaced by $\psi$ leads to
$$
\partial_t \psi +\nabla_x\cdot(u(x)\psi) +\mathcal{R}\cdot(n\times \nabla_xu \cdot \psi)=D_t\Delta_x \psi +D_r \mathcal{R}\cdot \left[\mathcal{R}\psi +\frac{\psi}{k_B T}\mathcal{R}U \right],
$$
where
$$
D_t=\frac{k_BT}{\zeta_t}, \quad D_r=\frac{k_BT}{\zeta_r}.
$$

This is exactly the equation derived for rod-like polymers with no inertial force. \\

\noindent {\bf Proof of Lemma \ref{jj}}.  The integration is over
$\omega$ from a tangent bundle with $(n, \omega)\in T \mathbb{S}^2$.
We fix $n\in \mathbb{S}^2$ arbitrarily, it suffices to evaluate
$$
\int _{ n\cdot \omega=0} e^{-\omega^2/a}(A\cdot \omega)\omega d_n\omega=\frac{a}{2}(Id-n\otimes n)A \int _{n\cdot \omega=0} e^{-\omega^2/a}d_n\omega.
$$
Let $K$ be a rotational operator in $\mathbb{R}^3$ such that
$$
n=Ke_3, \quad K\in SO(3),
$$
and use transform $\omega=K \omega'$, satisfying
$|\omega|^2=|\omega'|^2$.

Let $A=KB$, the transformation gives
\begin{align*}
\int _{ n\cdot \omega=0} e^{-\omega^2/a}(A\cdot \omega)\omega d_n\omega &=
\int _{ \omega_3'=0} e^{-(\omega')^2/a}(A\cdot K\omega')K\omega' d\omega' \\
& =K\int_{\omega_3'=0} e^{-((\omega_1')^2+(\omega_2')^2)/a}(b_1\omega_1'+b_2\omega_2')\omega'd\omega_1'd\omega_2'\\
&=\frac{a}{2}K\left(
                \begin{array}{c}
                  b_1 \\
                  b_2 \\
                  0 \\
                \end{array}
              \right)
\int_{\omega_3'=0} e^{-((\omega_1')^2+(\omega_2')^2)/a}d\omega_1'd\omega_2' \\
&=\frac{a}{2}K (B-e_3b_3) \int _{n\cdot \omega=0}
e^{-\omega^2/a}d_n\omega
\\
&=\frac{a}{2}(Id-n\otimes n)A \int _{n\cdot \omega=0} e^{-\omega^2/a}d_n\omega.
\end{align*}
Here we have used the fact that
$$e_3b_3=K^{-1}nb_3=ne_3^\top K^TA=K^{-1}n\otimes n A.
$$
\section{Concluding remarks}
In this work we have derived novel kinetic equations for
Dumbbell-like polymers as well as rod-like polymers. Inertial forces
are taken care of by an augmented environment in an extended configuration space. In the
case of rod-like polymers, the augmented space for orientation is
just a tangent bundle of the usual sphere. We have also shown that
the formal limit of the augmented equation recovers the usual
inertia-free kinetic models explored in literature.

\bigskip
\section*{Acknowledgments}
This research was partially supported by the National Science Foundation under the Kinetic FRG grant DMS07-57227 and by the Marie Curie Actions of the European Commission in the frame of the DEASE project (MEST-CT-2005-021122).


\end{document}